\documentclass[11pt, a4paper]{amsart}
\usepackage{amsfonts,amsmath, amsthm, amssymb, amscd}
\usepackage{latexsym}
\input{amssym.def}
\input{amssym}

\usepackage[dvips]{graphicx}
\usepackage{psfrag}

\newtheorem{theorem}{Theorem}[section]
\newtheorem{lemma}[theorem]{Lemma}
\newtheorem{cor}[theorem]{Corollary}
\newtheorem{prop}[theorem]{Proposition}
\newtheorem{dfn}[theorem]{{Definition}}

\newtheorem*{rmk}{{Remark}}

\numberwithin{equation}{section}
\newcommand {\Z}{\mathbb{Z}} 
\newcommand {\R}{\mathbb{R}} 

\DeclareMathOperator{\id}{id}
\DeclareMathOperator{\vol}{vol}

\DeclareMathOperator{\grad}{grad}
\DeclareMathOperator{\rank}{rank}
\DeclareMathOperator{\tr}{tr}

\begin{document}
\title[Noncompact asymptotically harmonic manifolds]{Noncompact asymptotically harmonic manifolds}
\author{Gerhard Knieper and Norbert Peyerimhoff }
\date{\today}
\address{Faculty of Mathematics,
Ruhr University Bochum, 44780 Bochum, Germany}
\email{gerhard.knieper@rub.de}
\address{Department of Mathematical Sciences, Durham University, Durham DH1 3LE, UK}
\email{norbert.peyerimhoff@dur.ac.uk}


\begin{abstract}
  In this article we consider asymptotically harmonic manifolds which are simply
  connected complete Riemannian manifolds without conjugate points such
  that all horospheres have the same constant mean curvature $h$. We prove
  the following equivalences for asymptotically harmonic manifolds $X$ 
  under the additional assumption that their curvature tensor together with
  its covariant derivative are uniformly bounded: (a) $X$ has rank
  one; (b) $X$ has Anosov geodesic flow; (c) $X$ is Gromov hyperbolic;
  (d) $X$ has purely exponential volume growth with volume entropy equals $h$. 
  This generalizes earlier results by G. Knieper for noncompact harmonic
  manifolds and by A. Zimmer for asymptotically harmonic manifolds
  admitting compact quotients.
\end{abstract}


\maketitle

\tableofcontents

\section{Introduction}

Let $(X,g)$ be a complete simply connected Riemannian manifold without
conjugate points and $SX$ its unit tangent bundle. For $v \in SX$ we
denote by $c_v: \R \to X$ the corresponding geodesic with $c_v'(0) =
v$ and $b_v: X\to \mathbb{R}$, $b_v(q) = \lim_{t \to \infty}
d(c_v(t),q) - t$ be the associated Busemann function.

Let $S_{v,r}$ and $U_{v,r}$ be the orthogonal Jacobi tensors along
$c_v$, defined by $S_{v,r}(0) = U_{v,r}(0) = \id$ and $S_{v,r}(r) = 0$
and $U_{v,r}(-r) = 0$. Note that we have $U_{v,r}(t) =
S_{-v,r}(-t)$. The stable and unstable Jacobi tensors $S_v$ and $U_v$
are defined as the Jacobi tensors along $c_v$ with initial conditions
$S_v(0) = U_v(0) = \id$ and $S_v'(0) = \lim_{r \to \infty}
S_{v,r}'(0)$ and $U_v'(0) = \lim_{r \to \infty} U_{v,r}'(0)$. We
define $U(v) = U_v'(0)$ and $S(v) = S_v'(0)$. For a general introduction
into Jacobi tensors see \cite{Kn1}.

Important for this paper will be the notion of rank which in
nonpositive curvature has been defined in \cite{BBE} as the dimension
of the parallel Jacobi fields along geodesics, and is one of the
central concepts in rigidity theory.  In the case of no conjugate
points it is due to Knieper \cite{Kn2} and generalizes this concept.

\begin{dfn}
  Let $(X,g)$ be a complete simply connected Riemannian manifold
  without conjugate points. For $v \in SX$ let $D(v) = U(v)-S(v)$ and we define
  $$
  \rank(v) = \dim (\ker D(v)) +1
  $$
  and
  $$
  \rank(X) = \min  \{ \rank(v) \mid  v \in SX \}
  $$
\end{dfn}

It is easy to see that the function $v \to \rank(v)$ is invariant
under the geodesic flow. 

As already observed in \cite{Kn2} the notion of $\rank$ is very
important in the study of {\em harmonic manifolds}. After Szabo's
proof \cite{Sz} of the Lichnerowicz conjecture for {\em compact} simply
connected harmonic manifolds, the classification of {\em noncompact}
harmonic manifolds is still wide open, even though there have been
interesting new developments in the last decade (see, e.g.,
\cite{RaSh,Ni,He}). In this paper we consider the more general class of
asymptotically harmonic manifolds, originally introduced by Ledrappier
\cite[Thm 1]{Le} in connection with rigidity of measures related to
the Dirichlet problem (harmonic measure) and the dynamics of the
geodesic flow (Bowen-Margulis measure).

\begin{dfn}
  An {\em asymptotically harmonic manifold} $(X,g)$ is a complete,
  simply connected Riemannian manifold without conjugate
  points such that for all $v \in SX$ we have $\tr U(v) = h$ for a
  constant $h \ge 0$.
\end{dfn}

Our first main result is the following:

\begin{theorem}\label{thm:Dconstant} Let $(X,g)$ be an asymptotically
  harmonic manifold such that $\vert R \Vert \le R_0$
  and $\vert \nabla R \Vert \le R_0'$ with suitable constants
  $R_0,R_0' > 0$. Then $v \mapsto \det D(v)$ is a constant function
  on $SX$. 
  
  Moreover, if $X$ has rank one, there exists $\rho > 0$ such that 
  $D(v) \ge \rho \cdot \id$ for all $v \in SX$.
\end{theorem}

For harmonic manifolds, this theorem is a consequence of the relation between
$\det D(v)$ and the volume density function (see
\cite[Cor. 2.5]{Kn2}). For asymptotically harmonic manifolds this theorem
was proved in \cite[Cor. 2.1]{HKS} under the additional condition of
strictly negative curvature bounded away from zero.  Zimmer
\cite[Proof of Prop. 3.3]{Zi} provides a proof under the additional
assumption of the existence of a compact quotient, using dynamical
arguments. The proof of the general case without negative curvature or
compact quotient requires new subtle estimates for second fundamental
forms of spheres and horospheres which are presented in Section 2 of
this article.

For the next result about asymptotic geometric and dynamical properties equivalent 
to the rank one condition we first need to introduce the notion of volume entropy.

\begin{dfn}
  The {\em volume entropy} $h_{vol}(X)$ of a connected Riemannian
  manifold $X$ is defined as
   \begin{equation} \label{eq:hvoldef} 
   h_{vol}(X) =  \limsup_{r \to \infty} \frac{\log \vol B_r(p)}{r}, 
   \end{equation}
   where $B_r(p) \subset X$ is the open ball of radius $r$ around $p \in X$.
 \end{dfn}
 
 Note that \eqref{eq:hvoldef} does not depend on the choice of reference point $p$ and 
 $h_{vol}(X)$ is therefore well defined.  
 
Theorem \ref{thm:Dconstant} is essential in the proof of our second main result.

\begin{theorem} \label{thm:equivalences} Let $(X,g)$ be an
  asymptotically harmonic manifold such that $\vert R \Vert \le R_0$
  and $\vert \nabla R \Vert \le R_0'$ with suitable constants
  $R_0,R_0' > 0$. Let $h \ge 0$ be the mean curvature of its
  horospheres, i.e. $h = \tr U(v)$. Then the following properties are
  equivalent.
  \begin{itemize}
  \item[(a)] $X$ has rank one.
  \item[(b)] $X$ has Anosov geodesic flow $\phi^t: SX \to SX$.
  \item[(c)] $X$ is Gromov hyperbolic.
  \item[(d)] $X$ has purely exponential volume growth with growth rate
    $h_{vol} = h$.
  \end{itemize}
\end{theorem}

This equivalence has been obtained in the case of {\em noncompact
  harmonic manifolds} by Knieper in \cite{Kn2}. In the case that $(X,g)$ is an asymptotically
harmonic manifold {\em with compact quotient}, this equivalence has
been derived by Zimmer \cite{Zi}. Since for harmonic manifolds the
curvature tensor and its covariant derivative are bounded (\cite[Props. 6.57 and 6.68]{Be}), the current 
article generalizes these results in both papers to asymptotically harmonic manifolds (without a compact
quotient condition).

In a subsequent article \cite{KnPe} we use the main results of this article to derive results about harmonic functions 
(solution of the Dirichlet problem at infinity and mean value property of harmonic functions at infinity) 
on rank one asymptotically harmonic manifolds.

\section{Manifolds without conjugate points: general results}

In this section, $(X,g)$ always denotes a complete simply connected
Riemannian manifold without conjugate points. Let $\pi: SX \to X$ be
the footpoint projection and $v \in SX$. The associated curvature
tensor $R_v(t): \phi^t(v)^\bot \to \phi^t(v)^\bot$ along $c_v$ is
defined by
$$ R_v(t)w = R(w,\phi^t(v))\phi^t(v). $$
The stable and unstable manifolds through $v \in SX$ are defined as
$W^s(v) = \{ - \grad b_v(q) \mid b_v(q) = 0 \}$ and $W^u(v) = \{ \grad
b_{-v}(q) \mid b_{-v}(q) = 0 \}$. The footpoint projections $\pi
W^s(v)$ and $\pi W^u(v)$ are level sets of Busemann functions and,
therefore, horospheres. Horospheres are usually denoted by $\mathcal
H$. Observe that $S(v)$ and $U(v)$ are the associated second
fundamental forms.

\subsection{A formula for the difference of second fundamental
  forms in horospheres}
 
Of importance is the following result which is based on an formula of E.
Hopf \cite[(7.2)]{Ho} for surfaces.

\begin{prop} \label{prop:Sdiff-formula} Let $\gamma: [0,1] \to W^s(v)$
  be a smooth curve with $\gamma(0) = v$ and $\gamma(1) = \widetilde
  v$. Let $e_1(s),\dots,e_{n-1}(s)$ be an orthonormal frame in
  ${\mathcal H} = \pi W^s(v)$ along $\beta = \pi \gamma$ which is
  parallel in ${\mathcal H}$ with the induced connection. Let
  $e_i(s,t)$ be the parallel translation along the geodesic
  $c_{\gamma(s)}$. Then we have
  \begin{equation} \label{eq:Stvr'-Svr'} 
  S_{\widetilde v,r}'(0) - S_{v,r}'(0) = \int_0^1 
  \int_0^r S_{\gamma(s),r}^{\ast}(t) \left( \frac{\partial}{\partial s} 
  R_{\gamma(s)}(t) \right) S_{\gamma(s),r}(t)dt\; ds 
  \end{equation}
  and
  \begin{equation} \label{eq:Utvr'-Uvr'} 
  U_{\widetilde v,r}'(0) - U_{v,r}'(0) = - \int_0^1 
  \int_{-r}^0 U_{\gamma(s),r}^{\ast}(t) \left( \frac{\partial}{\partial
      s} R_{\gamma(s)}(t) \right) U_{\gamma(s),r}(t)dt\; ds,
  \end{equation} 
  where all tensors are expressed with respect to the frame
  $e_1(s,t),\dots,e_{n-1}(s,t)$.
\end{prop}

\begin{proof}
  We only prove \eqref{eq:Stvr'-Svr'}, the second identity is proved
  analogously. We start with the Jacobi equation
  $$
  S_{ \gamma(s),r}''(t)+ R_{ \gamma(s)}(t) S_{ \gamma(s),r}(t) =0
  $$
  and define
  $$
  Z_{ \gamma(s),r}(t) = \frac{\partial}{\partial s}S_{\gamma(s),r}(t).
  $$
  Then we have
  \begin{eqnarray*}
    Z_{\gamma(s),r}''(t) &=&\frac{\partial}{\partial s} 
    \frac{\partial^2}{\partial^2 t} S_{\gamma(s),r}(t) =
    -\frac{\partial}{\partial s} \left(R_{\gamma(s)}(t) S_{\gamma(s),r}(t)\right) \\
    &=& - \left(\frac{\partial}{\partial s} R_{\gamma(s)}(t) \right) 
    S_{\gamma(s),r}(t) - R_{\gamma(s)}(t) \left(\frac{\partial}{\partial s} 
      S_{\gamma(s),r}(t) \right), 
  \end{eqnarray*}
  and therefore,
  $$
  Z_{\gamma(s),r}''(t) = - R_{\gamma(s)}(t) Z_{\gamma(s),r}(t) -
  \left( \frac{\partial}{\partial s} R_{\gamma(s)}(t) \right) S_{
    \gamma(s),r}(t).
  $$
  Differentiating the Wronskian of $Z_{\gamma(s),r}$ and
  $S_{\gamma(s),r}$, we obtain
  \begin{multline*}
  \frac{\partial}{\partial t} \left( (Z^{\ast}_{\gamma(s),r})'(t)
    S_{\gamma(s),r}(t) - Z^{\ast}_{\gamma(s),r}(t) S'_{\gamma(s),r}(t) \right) =\\
  (Z^{\ast}_{\gamma(s),r})''(t) S_{\gamma(s),r}(t) -
  Z^{\ast}_{\gamma(s),r}(t) S_{\gamma(s),r}''(t) = \\
  -Z^{\ast}_{ \gamma(s),r}(t) R_{ \gamma(s)}(t) S_{ \gamma(s),r}(t) -
  S^{\ast}_{ \gamma(s),r}(t) \left( \frac{\partial}{\partial s} R_{
      \gamma(s)}(t) \right) S_{ \gamma(s),r}(t) \\ + Z^{\ast}_{
    \gamma(s),r}(t) R_{ \gamma(s)}(t) S_{ \gamma(s),r}(t) = -
  S^{\ast}_{ \gamma(s),r}(t) \left( \frac{\partial}{\partial s} R_{
      \gamma(s)}(t) \right) S_{ \gamma(s),r}(t).
  \end{multline*}
  Integration with respect to $t$ from $0$ to $r$ yields
  $$
  \frac{\partial}{\partial s} S'_{\gamma(s),r}(0) = (Z^{\ast}_{
    \gamma(s),r})'(0) = \int_0^r S^{\ast}_{\gamma(s),r}(t) \left(
    \frac{\partial}{\partial s} R_{ \gamma(s)}(t) \right)
  S_{\gamma(s),r}(t) dt.
  $$
  Integration with respect to $s$ from $0$ to $1$ leads finally to
  $$
  S'_{\widetilde v,r}(0) - S'_{v,r}(0) = \int_0^1 \int_0^r
  S^{\ast}_{\gamma(s),r}(t) \left( \frac{\partial}{\partial s}
    R_{\gamma(s)}(t) \right) S_{\gamma(s),r}(t) dt \; ds,
  $$
  proving \eqref{eq:Stvr'-Svr'} after interchanging the integrals.
\end{proof}

In order to make use of the formulas in Proposition
\ref{prop:Sdiff-formula}, we need to have estimates for $\Vert
S_{\gamma(s),r} \Vert$, $\Vert U_{\gamma(s),r} \Vert$ and $\Vert
\frac{\partial}{\partial s} R_{\gamma(s)}(t) \Vert$. These estimates are
derived in the following two subsections.

\subsection{Estimates for $\Vert S_{v,r} \Vert$ and $\Vert U_{v,r} \Vert$}
We recall the following facts from \cite[Chapter 1, Cor. 2.12 and Lemma 2.17]{Kn1}
(choosing $r = \infty$ there):

\begin{lemma} \label{lem:AAm1} Assume that there exists a constant $R_0
  > 0$ such that $-R_0 \id \le R_v(t)$ for all $v \in SX$ and $t \in
  \R$. Let $A_v$ be the orthogonal Jacobi tensor along $c_v$ with
  $A_v(0) = 0$ and $A_v'(0) = \id$. Then we have
  $$ - \sqrt{R_0} \le A_v'(t) A_v^{-1}(t) \le \sqrt{R_0} \coth (t \sqrt{R_0}) $$
  for all $t > 0$. Furthermore, we have
  $$ -\sqrt{R_0} \le S_v'(0) \le U_v'(0) \le \sqrt{R_0} $$
  for all $v \in SX$.
\end{lemma}

Note that $A_v$ and $S_{v,r}$ are related by $S_{v,r}(t) = A_v(r-t)A_v^{-1}(r)$. 
Therefore, Lemma \ref{lem:AAm1} has the following consequence.

\begin{cor} \label{cor:CR0r0T} Let $r_0 > 1$ and $T \le r_0$. If
  $\Vert R_v(t) \Vert \le R_0$ for all $v \in SX$ and $t \in \R$ with
  a constant $R_0 > 0$, we have for all $r \ge r_0$
  $$ \Vert S_{v,r}(t) \Vert \le C_1(R_0,r_0,T) \qquad 
  \text{for all $0 \le t \le T$}, $$ 
  with $C_1(R_0,r_0,T) > 0$ only depending on $r_0$, $R_0$ and $T$.
\end{cor}

\begin{proof}
  We conclude from Lemma \ref{lem:AAm1} for all $r \ge r_0$,
  $$ \Vert S_{v,r}'(0) \Vert = \Vert A_v'(r) A_v^{-1}(r) \Vert \le 
  \sqrt{R_0} \coth (r_0 \sqrt{R_0}). $$ 

  Let $y(t) = (y_1(t),y_2(t))^\top$ with $y_1(t) = S_{v,r}(t)$ and
  $y_2(t) = S_{v,r}'(t)$. Then
  $$ y'(t) = \begin{pmatrix} y_1'(t) \\ y_2'(t) \end{pmatrix} = 
  \begin{pmatrix} 0 & 1 \\ -R_v(t) & 0 \end{pmatrix} \begin{pmatrix}
    y_1(t) \\ y_2(t) \end{pmatrix} = C(t) y(t), $$ 
  i.e.,
  $$ y(t) = \exp\left(\int_0^t C(s) ds\right) y(0). $$
  Note that $\Vert C(t) \Vert \le \sqrt{{R_0}^2 + 1}$ and $\Vert y(0)
  \Vert^2 \le 1 + R_0 \coth^2 (r_0 \sqrt{R_0})$. This yields
  $$ \Vert S_{v,r}(t) \Vert \le \Vert y(t) \Vert \le 
  \exp\left(T \sqrt{{R_0}^2 + 1}\right) \sqrt{1 + R_0 \coth^2 (r_0
    \sqrt{R_0})} $$ 
  for all $0 \le t \le T$, finishing the proof.
\end{proof}

Next, we present some useful Jacobi tensor identities.

\begin{lemma} \label{central} For all $v \in SX$ and $ t < r$ we have
  \begin{equation}\label{c1}
    S_{\phi^t(v),r-t}'(0) =  S_{v,r}'(t) S_{v,r}^{-1}(t), 
  \end{equation}
  and
  \begin{eqnarray}\label{c2}
    U_{\phi^t(v)}'(0) - S_{\phi^t(v), r-t}'(0) &=&
    (U_v^{\ast})^{-1}(t)(U_v'(0) -S_{v, r}'(0)) S_{v, r}^{-1}(t)\nonumber \\
    &=&(S_{v, r}^{\ast})^{-1}(t)(U_v'(0) -S_{v, r}'(0)) U_v^{-1}(t).
  \end{eqnarray}
  Furthermore,
  \begin{equation}\label{c3}
    U_{\phi^t(v)}'(0) - S_{\phi^t(v),r-t}'(0) = 
    (S^{\ast}_{v,r})^{-1}(t) \left(\int\limits_{-\infty}^{t} 
      (S_{v,r}^{\ast}S_{v,r})^{-1}(u) du \right)^{-1} S_{v,r}^{-1}(t).
  \end{equation}
\end{lemma}

\begin{proof}
  Notice first that
  \begin{equation} \label{eq:Stransform}
  S_{\phi^t v,x}(y) = S_{v,t+x}(y+t) S_{v,t+x}^{-1}(t),
  \end{equation}
  since, for fixed $x $ and $t$, both sides define Jacobi tensors in
  $y$ which agree at $y=0$ and $y=x$. Differentiating at $y= 0$ yields
  for $x = r-t$ the first identity \eqref{c1}. Using the fact that the
  Wronskian of two Jacobi tensors is constant, we have
  \begin{multline*}
  W(U_v,S_{v,r})(t) = (U_v^{\ast})'(t) S_{v,r}(t) - U_v^{\ast}(t)
  S_{v, r}'(t) = \\ W(U_v,S_{v, r})(0) =U_v'(0) -S_{v,r}'(0).
  \end{multline*}
  This yields
  $$
  (U_v'(t) U_v^{-1}(t))^{\ast}- S_{v,r}'(t) S_{v,r}^{-1}(t) =
  (U_v^{\ast})^{-1}(t) (U_v'(0) -S_{v,r}'(0)) S_{v,r}^{-1}(t).
  $$
  Since
  $$
  (U_v'(t) U_v^{-1}(t))= U_{\phi^t(v)}'(0) \; \; \text{and} \; \;
  S_{v,r}'(t) S_{v,r}^{-1}(t) = S_{\phi^t(v),r-t}'(0)
  $$
  are symmetric, we obtain the first and second identity of \eqref{c2}.

  To prove the last assertion we note that for $0 \le t \le r$ we have
  (see \cite[(7.8)]{Kn2})
  $$
  \left(\int\limits_{-\infty}^{t} (S_{v,r}^{\ast}S_{v,r})^{-1}(u) du
  \right)^{-1} S_{v,r}(t)^{-1}U_v(t)= U_v'(0) - S_{v,r}'(0).
  $$ 
  Inserting this into \eqref{c2} yields \eqref{c3}.
\end{proof}

Recall from the introduction that $S(v) = S_v'(0)$ and $U(v)= U_v'(0)$. 
A key role plays the positive symmetric operators
$$ D(v) = U(v) - S(v), $$
since their kernels determine the rank of the manifold $X$.

\begin{prop} \label{prop:SUest} Assume there exists $R_0 > 0$ such
  that $\Vert R_v(t) \Vert \le R_0$ for all $t \in \R$. Then we have the 
  following estimates for $S_{v,r}$ and $U_{v,r}$.
  \begin{itemize}
  \item[(a)] There exists $a_1 = a_1(R_0)$ such that for all $r > 1$
    and $t \ge 0$,
    $$ \Vert S_{v,r}(-t) \Vert \le a_1 e^{\sqrt{R_0} t}, \qquad 
    \Vert U_{v,r}(t) \Vert \le a_1 e^{\sqrt{R_0} t}. $$
  \item[(b)] Under the additional assumption $D(\phi^t(v)) \ge \rho
    \cdot \id$ for all $t \in \R$ and some constant $\rho > 0$, there
    exists $a_2 = a_2(R_0,\rho)$ such that for all $r > 1$ and $0 \le
    t < r$,
    $$ \Vert S_{v,r}(t) \Vert \le a_2 e^{-\frac{\rho}{2} t}, \qquad 
    \Vert U_{v,r}(-t) \Vert \le a_2 e^{-\frac{\rho}{2} t}. $$
  \end{itemize}
\end{prop}

\begin{proof}
  Rauch's comparison estimate (see, e.g., \cite[Chapter 1, Prop. 2.11]{Kn1}) 
  implies that $\Vert A(t) x \Vert / \sinh \sqrt{R_0} t$
  is monotone decreasing. Since $S_{v,r}(-t) = A_v(r+t) A_v^{-1}(r)$
  we conclude
  $$ \Vert S_{v,r}(-t) \frac{A_v(r)x}{\Vert A_v(r)x \Vert} \Vert
  = \frac{\Vert A(r+t)x \Vert}{\Vert A(r)x \Vert} \le \frac{\sinh
    \sqrt{R_0}(r+t)}{\sinh \sqrt{R_0} r} \le a_1(R_0) e^{\sqrt{R_0} t}. $$ 
  This together with $U_{v,r}(t) = S_{-v,r}(-t)$ proves (a).

  Using the monotonicity $S_{w,r}'(0) \nearrow S_w'(0)$, we have by
  assumption
  $$
  U_{\phi^t(v)}'(0) - S_{\phi^t(v),r-t}'(0) \ge U_{\phi^t(v)}'(0) -
  S_{\phi^t(v)}'(0) = D(\phi^t(v)) \ge \rho \id.
  $$
  Using \eqref{c3}, this yields for all $x \in (\phi^t v)^\perp$ with
  $\|x\| =1$,
  \begin{eqnarray*}
    \rho &\le& \left \langle \left( \int\limits_{-\infty}^{t} 
        (S_{v,r}^{\ast}S_{v,r}^{-1}(u) du \right)^{-1}  
      S_{v,r}^{-1}(t)x, S_{v,r}^{-1}(t)x \right \rangle\\
    &\le& \left\| 
      \left(\int\limits_{-\infty}^{t} (S_{v,r}^{\ast}S_{v, r})^{-1}(u) du \right)^{-1}       
      \right\| \cdot \left\| S_{v, r}^{-1}(t)x \right\|^2.
  \end{eqnarray*}
  Furthermore, we have
  \begin{equation*}
    \begin{split}
      &\left\| \left(\int\limits_{-\infty}^{t}
          (S_{v,r}^{\ast}S_{v,r})^{-1}(u) du \right)^{-1} \right\| =\\
      &\quad \quad \frac{1} { \min \left \{\int\limits_{-\infty}^{t}
          \langle (S_{v,r}^{\ast}S_{v,r})^{-1}(u)y_u, y_u \rangle
          du:\; y \in v^\perp, \|y\| =1 \right\} }.
    \end{split} 
  \end{equation*}
  Therefore,
  \begin{equation*}
    \begin{split}
      \rho \min &\left\{\int\limits_{-\infty}^{t} \langle
        (S_{v,r}^{\ast})^{-1}(u)y_u, (S_{v,r}^{\ast})^{-1}(u)y_u \rangle du
        :\; y \in v\perp, \|y\| =1) \right\}\\
      & \quad \quad \le \|S_{v,r}^{-1}(t)x \|^2
    \end{split}
  \end{equation*}
  for all $x \in (\phi^t v)^\perp$ with $\|x\| =1$. Defining
  \begin{eqnarray*}
    \varphi(u) &:= &\min \left\{\|(S_{v,r}^{\ast})^{-1}(u)y \|^2:\; 
      y \in ( \phi^u v)^\perp, \|y\| = 1 \right\}\\
    &=& \min \left\{\|S_{v,r}^{-1}(u)y \|^2:\; 
      y \in (\phi^u v) ^\perp, \|y\| = 1 \right\},
  \end{eqnarray*}
  we obtain
  $$
  \rho \int\limits_{0}^{t} \varphi(u)du \le \rho
  \int\limits_{-\infty}^{t} \varphi(u)du \le \varphi(t)
  $$
  and, hence,
  $$
  \rho F(t) \le F'(t)
  $$
  for $F(t) := \int\limits_{0}^{t} \varphi(u)du$. This means that
  $\rho \le (\log F)'(t)$, which implies $F(t) \ge \frac{F(1)}{e}
  e^{\rho t}$ for all $t \ge 1$ and, therefore,
  \begin{equation} \label{eq:phit}
  \varphi(t)= F'(t) \ge \rho F(t)\ge \frac{\rho F(1)}{e} e^{\rho t} 
  \end{equation}
  for all $r > t \ge 1$. Corollary \ref{cor:CR0r0T} implies for $0 \le
  t \le 1 < r$
  \begin{eqnarray*} 
  \varphi(t) &=& \min \left\{\|S_{v,r}^{-1}(t)y \|^2:\; 
    y \in (\phi^t v) ^\perp, \|y\| = 1 \right\}\\ &=& \frac{1}{\Vert
    S_{v,r}(t) \Vert} \ge \frac{1}{C_1(R_0,1,1)},
  \end{eqnarray*}
  i.e., $F(1) \ge \frac{1}{C_1(R_0,1,1)}$. Plugging this into \eqref{eq:phit},
  we obtain
  $$
  \varphi(t) \ge \frac{\rho}{e C_1(R_0,1,1)} e^{\rho t} $$ 
  for all $r > t \ge 1$. Choosing $a_2 = \left(
    \frac{C_1(R_0,1,1)}{\min\{ \rho/e,e^{-\rho} \} } \right)^{1/2}$,
  this implies that we have
  $$ \frac{\Vert S_{v,r}^{-1}(t)y \Vert^2}{\Vert y \Vert^2} \ge \varphi(t) \ge
  \frac{1}{C_1(R_0,1,1)}\min\{ \frac{\rho}{e}, e^{-\rho}
  \}e^{\rho t} = \frac{1}{{a_2}^2} e^{\rho t} $$ 
  for all $r > t \ge 0$ and $y \in (\phi^t v)^\bot$, $y \neq 0$.
  Since $S_{v,r}(t): (\phi^tv)^\perp \to (\phi^tv)^\perp$ is an
  isomorphism, we obtain for all $x \in (\phi^tv)^\perp$, $r > 1$ and $t \in
  [0,r)$
  $$
  \|S_{v, r}(t)x \| \le a_2 e^{-\frac{\rho}{2} t} \|x \|,
  $$
  finishing the proof of (b).
\end{proof}

\begin{rmk}
  The special case of Proposition \ref{prop:SUest}(b) for stable and
  unstable Jacobi tensors was obtained
  by Bolton (see \cite[Lemma 2]{Bo}). 
\end{rmk}

The following corollary summarizes the facts which we will need further on 
in this chapter.

\begin{cor} \label{cor:SUbeta} Let $\Vert R_v(t) \Vert \le R_0$ for
  all $v \in SX$ and $t \in \R$ with a constant $R_0 > 0$. Let
  $\gamma: [0,1] \to W^s(v)$ be a smooth curve and $\rho > 0$ such that 
  $$ D(\phi^t(\gamma(s))) \ge \rho \cdot \id $$
  for all $s \in [0,1]$ and $t \in \R$. Then there exists a function
  $b: \R \to (0,\infty)$, only depending on $R_0$ and $\rho$, such
  that we have for all $r > 1$ and all $-\infty < t < r$,
  \begin{equation} \label{eq:SUest} 
  \Vert S_{\gamma(s),r}(t) \Vert \le b(t), \qquad 
  \Vert U_{\gamma(s),r}(-t) \Vert \le b(t). 
  \end{equation}
  For $t \ge 0$ we have
  \begin{equation} \label{eq:b(t)} 
  b(t) \le a_2 e^{-\frac{\rho}{2} t}. 
  \end{equation}
  Moreover, if $\beta = \pi \gamma$ and $\beta_t = \pi (\phi^t
  \gamma)$, we have
  \begin{eqnarray}
  \Vert \beta_t'(s) \Vert &\le& b(t)\, \Vert \beta'(s) \Vert \label{eq:betatbeta} \\
  \Vert (\phi^t \gamma)'(s) \Vert &\le& b(t)\sqrt{1 + R_0}\, \Vert \beta'(s) \Vert \label{eq:phitgamest}
  \end{eqnarray}
  for all $s \in [0,1]$ and $t \in \R$. 
\end{cor}

\begin{proof}
  The inequalities \eqref{eq:SUest} and \eqref{eq:b(t)} are straightforward consequences
  of Proposition \ref{prop:SUest}. The same inequalities hold also for
  the stable and unstable Jacobi tensors $S_{\gamma(s)}$ and
  $U_{\gamma(s)}$. Note that $J_s(t) = \beta_t'(s) =
  \frac{\partial}{\partial s} c_{\gamma(s)}(t)$ is the stable Jacobi
  field along $c_{\gamma(s)}$ with initial values
  $$ J_s(0) = \beta'(s) \quad \text{and} \quad 
  J_s'(0) = S_{\gamma(s)}'(0) J_s(0). $$
  Then $J_s(t) = S_{\gamma(s)}(t) (J_s(0))_t$, which implies
  $$ \Vert \beta_t'(s) \Vert \le \Vert S_{\gamma(s)}(t) 
  ( \beta'(s) )_t \Vert \le b(t) \Vert \beta'(s) \Vert. $$
  Furthermore we have
  $$ \Vert (\phi^t \gamma)'(s) \Vert^2 = \Vert \frac{d}{ds} \beta_t(s) \Vert^2 + \Vert  \frac{D}{ds} \phi^t \gamma(s) \Vert^2
  = \Vert \frac{d}{ds} \beta_t(s) \Vert^2 + \Vert  \nabla_{\beta_t'(s)} \phi^t\gamma(s) \Vert^2. $$
  Since $\nabla_{\beta_t'(s)} \phi^t\gamma(s)$ is the second fundamental form of the horosphere 
  $\pi W^s(\phi^t(\gamma(s)))$, we have with Lemma \ref{lem:AAm1}
  $$ \Vert (\phi^t \gamma)'(s) \Vert^2 = \Vert \frac{d}{ds} \beta_t(s) \Vert^2 + \Vert  S_{\phi^t \gamma(s)}'(0) \beta_t'(s) \Vert^2
  \le b(t)^2 (1 + R_0) \Vert \beta'(s) \Vert^2. $$
  This implies \eqref{eq:phitgamest}.
\end{proof}

\subsection{Estimate for $\Vert \frac{\partial}{\partial s}
  R_{\gamma(s)}(t) \Vert$}

Our next goal is to derive an estimate for $\Vert
\frac{\partial}{\partial s} R_{\gamma(s)}(t) \Vert$ in terms of
$\beta'(s)$. Henceforth, we assume that the curvature tensor and
its covariant derivative of $X$ are bounded, i.e.,  
$$
\Vert R \Vert \le R_0 \quad \text{and} \quad \Vert \nabla R \Vert \le R_0' 
$$
with suitable constants $R_0, R_0' > 0$. Moreover, let $\gamma: [0,1]
\to W^s(v)$ denote a smooth curve such that
$$ D(\phi^t(\gamma(s))) \ge \rho \cdot \id $$
for all $s \in [0,1]$ and $t \in \R$ with a suitable constant $\rho > 0$. 
Let $e_i(s)$ and $e_i(s,t)$ be defined as in Proposition
\ref{prop:Sdiff-formula} and $\beta = \pi \gamma$ and $\beta_t =
\pi(\phi^t \gamma)$.

\begin{lemma} \label{lem:dsei}
  Let $r > 1$. Then there exists a constant $C_2(R_0,\rho,r)$, only
  depending on $R_0, \rho$ and $r$, such that
  $$ \left\Vert \frac{D}{ds} e_i(s,t) \right\Vert \le C_2(R_0,\rho,r) 
  \Vert \beta'(s) \Vert \qquad \text{for all $t \in (-r,r)$}. $$
\end{lemma}

\begin{proof}
  First of all, note that the second fundamental form of all
  horospheres is bounded by $\sqrt{R_0}$. Let $N$ be a unit normal
  vector field of ${\mathcal H} = \pi W^s(v)$. Since $e_i$ is parallel
  in ${\mathcal H}$ with respect to the induced connection, we have
  $$ \frac{D}{ds} e_i(s) = \left\langle \frac{D}{ds} e_i(s), 
    (N \circ \beta)(s) \right\rangle \; (N \circ \beta)(s). $$
  Therefore,
  \begin{eqnarray}
    \left\Vert \frac{D}{ds} e_i(s) \right\Vert^2 &=& \left\langle \frac{D}{ds}
      e_i(s), (N \circ \beta)(s) \right\rangle^2 = 
    \left\langle e_i(s), \frac{D}{ds} N \circ \beta(s) \right\rangle^2 
    \nonumber \\
    &\le& \Vert e_i(s) \Vert^2\; \Vert \nabla N \circ \beta(s) \Vert^2\; 
    \Vert \beta'(s) \Vert^2 \nonumber \\
    &\le& R_0\; \Vert \beta'(s) \Vert^2. \label{eq:Aest}  
  \end{eqnarray}
  Let $P_{\gamma(s)}^t$ be the parallel transport along $c_{\gamma(s)}$. Define
  $$ f(s,t) = \left\Vert \frac{D}{ds} e_i(s,t) \right\Vert = 
  \left\Vert \frac{D}{ds} P_{\gamma(s)}^t e_i(s) \right\Vert. $$
  Differentiation yields
  \begin{equation} \label{eq:df2}
    \frac{\partial}{\partial t} f^2(s,t) = 2 \left\langle \frac{D}{dt} 
    \frac{D}{ds} P_{\gamma(s)}^t e_i(s), \frac{D}{ds} e_i(s,t) \right\rangle. 
  \end{equation}
  Note that
  \begin{eqnarray*}
    \frac{D}{dt} \frac{D}{ds} P_{\gamma(s)}^t e_i(s) &=&
    \underbrace{\frac{D}{ds} \frac{D}{dt} P_{\gamma(s)}^t e_i(s)}_{=0}
    + R\left(c_{\gamma(s)}'(t),\frac{\partial}{\partial s}
      c_{\gamma(s)}(t) \right)e_i(s,t) \\
    &=& R\left(c_{\gamma(s)}'(t),\frac{\partial}{\partial s}
      c_{\gamma(s)}(t) \right)e_i(s,t).
  \end{eqnarray*} 
  Plugging this into \eqref{eq:df2} we conclude
  $$ \left\vert \frac{\partial f}{\partial t}(s,t) \right\vert \le R_0 \; 
  \Vert c_{\gamma(s)}'(t) \Vert \; \left\Vert
    \frac{\partial}{\partial s} c_{\gamma(s)}(t) \right\Vert \; \Vert
  e_i(s,t) \Vert = R_0 \; \Vert \beta_t'(s) \Vert,
  $$
  which implies
  \begin{eqnarray*}
    f(s,t) &\le& f(s,0) + \int_{\min\{0,t\}}^{\max\{0,t\}} 
    \left\vert \frac{\partial f}{\partial t}(s,\tau) \right\vert d\tau \\
    &\le& \left\Vert \frac{D}{ds} e_i(s) \right\Vert + R_0 \int_{-r}^r \Vert
    \beta'(s) \Vert d\tau \\
    &\stackrel{\eqref{eq:Aest}}{\le}& 
    \sqrt{R_0} \Vert \beta'(s) \Vert + R_0 \int_{-r}^r \left\Vert S_{\gamma(s)}(\tau)
      (\beta'(s))_\tau \right\Vert d\tau \\
    &\le& \sqrt{R_0} \; \Vert \beta'(s) \Vert + R_0  \int_{-r}^r b(t) dt \; 
    \Vert \beta'(s) \Vert. 
  \end{eqnarray*}
  This finishes the proof.
\end{proof}

The estimate for $\Vert \frac{\partial}{\partial s} R_{\gamma(s)}(t)
\Vert$ is derived from the components. The $(i,j)$-th component of
$R_{\gamma(s)}(t)$ is
$$ 
\langle R_{\gamma(s)}e_i(s,t), e_j(s,t) \rangle = \langle
R(e_i(s,t),\phi^t(\gamma(s))) \phi^t(\gamma(s)), e_j(s,t) \rangle.
$$
This implies that we have
\begin{multline*}
\left( \frac{\partial}{\partial s} R_{\gamma(s)}(t) \right)_{i,j} = 
\left\langle \frac{D}{ds} \left( R_{\gamma(s)}(t)e_i(s,t) \right),
  e_j(s,t) \right\rangle + \\ \left\langle R_{\gamma(s)}(t) e_i(s,t),
  \frac{D}{ds} e_j(s,t) \right\rangle.
\end{multline*}
Using
\begin{multline*}
  \nabla_J (R(Z,W)W) = \\
  (\nabla_JR)(Z,W)W + R(\nabla_JZ,W)W + R(Z,\nabla_JW)W + R(Z,W)\nabla_JW
\end{multline*}
and the bounds $\Vert R \Vert \le R_0$ and $\Vert \nabla R \Vert \le R_0'$,
we obtain
\begin{multline*}
  \left\Vert \frac{D}{ds} \left( R_{\gamma(s)}(t) e_i(s,t) \right) \right\Vert
  \le \\
  R_0'\; \Vert\; \beta_t'(t) \Vert + 
  R_0 \left( \left\Vert \frac{D}{ds} e_i(s,t) \right\Vert + 
  2 \left\Vert \frac{D}{ds} \left( \phi^t(\gamma(s) \right) \right\Vert \right)
  \le \\
  R_0'\; b(t)\; \Vert \beta'(s) \Vert + 
  R_0 \left( C_2(R_0,\rho,r) \Vert \beta'(s) \Vert + 
  2 \left\Vert \frac{D}{ds} \left( \phi^t(\gamma(s) \right) \right\Vert 
  \right),  
\end{multline*}
where we used \eqref{eq:betatbeta} and Lemma \ref{lem:dsei}. Since
$\frac{D}{ds} \phi^t(\gamma(s)) = \nabla_{\beta_t'(s)}
\phi^t(\gamma(s))$ is the second fundamental form of the horosphere
$\pi W^s(\phi^t(\gamma(s)))$
which is bounded in norm by $\sqrt{R_0} \Vert \beta_t'(s) \Vert$, 
we finally obtain
$$ \left\Vert \frac{D}{ds} \left( R_{\gamma(s)}(t) e_i(s,t) \right)
  \right\Vert \le C_3(R_0,R_0',\rho,r) \Vert \beta'(s) \Vert
$$
with $C_3(R_0,R_0',\rho,r) = R_0' b(t) + R_0 C_2(R_0,\rho,r) + 2 R_0^{3/2}
b(t)$.  This implies that
\begin{equation} \label{eq:dsRijest}
\left\vert \left( \frac{\partial}{\partial s} R_{\gamma(s)}(t) \right)_{i,j}   
\right\vert \le \left( C_3(R_0,R_0',\rho,r) + R_0 C_2(R_0,\rho,r) 
\right) \Vert \beta'(s) \Vert. 
\end{equation}

\subsection{An estimate for the difference of second fundamental forms
  in horospheres}

Combining the results in the first three subsections, we are now able
to prove the following result.

\begin{theorem} \label{thm:SUdistconv} Let $(X,g)$ be a complete
  simply connected Riemannian manifold without conjugate
  points. Assume that $\Vert R_0 \Vert \le R_0$ and $\Vert \nabla R
  \Vert \le R_0'$ with suitable constants $R_0, R_0' > 0$. Let
  $\gamma: [0,1] \to W^s(v)$ be a smooth curve and $\beta = \pi
  \gamma$. Assume that $D(\phi^t(\gamma(s))) \ge \rho \cdot \id$ for
  all $s \in [0,1]$ and $t \in \R$ and some constant $\rho > 0$. Let
  $r > 1$. Then there exists a constant $C_5(R_0,R_0',\rho,r) > 0$,
  only depending on $R_0,R_0',\rho$ and $r$, such that
  $$ \Vert S_{\gamma(1),r}'(0) - S_{\gamma(0),r}'(0) \Vert \le 
  C_5(R_0,R_0',\rho,r)\; \ell(\beta) $$
  and
  $$ \Vert U_{\gamma(1),r}'(0) - U_{\gamma(0),r}'(0) \Vert \le 
  C_5(R_0,R_0',\rho,r)\; \ell(\beta), $$
  where $\ell(\beta)$ denotes the length of the curve $\beta$. 
\end{theorem}

\begin{proof}
  We only give the proof of the second estimate, the first estimate is
  proved analogously. Let $v = \gamma(0)$ and $\widetilde v =
  \gamma(1)$. Inequality \eqref{eq:dsRijest} implies that there is
  a constant $C_4 = C_4(R_0,R_0',\rho,r) > 0$ such that
  $$ \left\Vert \frac{\partial}{\partial s} R_{\gamma(s)}(t) \right\Vert 
  \le C_4 \Vert \beta'(s) \Vert. $$
  We conclude from Proposition \ref{prop:Sdiff-formula} and
  Corollary \ref{cor:SUbeta} that
  \begin{eqnarray*}
    \Vert U_{\gamma(1),r}'(0) - U_{\gamma(0),r}'(0) \Vert &\le& \int_0^1 \int_{-r}^0
    \Vert U_{\gamma(s),r}(t) \Vert^2 \; \left\Vert \frac{\partial}{\partial s} 
      R_{\gamma(s)}(t) \right\Vert dt \; ds \\
    &\le& \int_0^1 C_4 \int_{-r}^0 b(-t)^2 dt\; \Vert \beta'(s) \Vert ds \\
    &\le& C_5(R_0,R_0',\rho,r) \; \ell(\beta)
  \end{eqnarray*}
  with $C_5(R_0,R_0',\rho,r) = C_4 \int_{-r}^r b(t)^2 dt$. 
\end{proof}

\section{The function $\det D(v)$ is constant}

From now on, we assume that $(X,g)$ is asymptotically harmonic. Recall
that we introduced the positive symmetric operator $D(v) = U(v) -
S(v)$. Our aim is to prove Theorem \ref{thm:Dconstant} in the Introduction.

Note first that in the case of asymptotically harmonic manifolds the
stable and unstable Jacobi tensors are continuous in the sense of
\cite[p. 242]{Es}. This property is also called {\em continuous
  asymptote}.

\begin{lemma} \label{lem:contasym}
  Let $(X,g)$ be an asymptotically harmonic manifold. Then $v \mapsto
  U(v)$ and $v \mapsto S(v)$ are continuous maps on $SX$.
\end{lemma}

\begin{proof} Since $U_{v,t}'(0) - U_v'(0)$ is positive, we have for
  $t > 0$
  $$ \Vert U_{v,t}'(0) - U_v'(0) \Vert \le \tr (U_{v,t}'(0) - U_v'(0)) = 
  \tr(U_{v,t}'(0)) - h. $$ 
  Since $\tr(U_{v,t}'(0))$ converges pointwise monotonically to $h$
  as $t \to \infty$, we conclude from Dini that the convergence is
  uniformly on all compact subsets of $SX$. Since the maps $v \mapsto
  U_{v,t}'(0)$ is continuous for all $t > 0$ and $U_{v,t}'(0) \to
  U_v'(0) = U(v)$ uniformly on compact sets, we conclude continuity of
  $v \mapsto U(v)$. The continuity of $v \mapsto S(v)$ follows
  immediately from $S(v) = -U(-v)$.
\end{proof}

As a start, it is easy to see that $\det D(v) = \det D(-v)$:
$$ \det D(-v) = \det (U(-v) - S(-v)) = \det (-S(v) + U(v)) = \det D(v). $$ 
Now we work towards the result that $\det D(v)$ is constant on all of
$SX$.

\subsection{$\det D(v)$ is constant along the geodesic flow}
\label{sec:constflow}

The arguments in this section follow the arguments given in the proof
of \cite[Lemma 2.2]{HKS}.

\begin{prop} \label{prop:detconst}
  Let $(X,g)$ be asymptotically harmonic. Then for all $v \in SX$, the map 
  $t \mapsto \det (D(\phi^t v))$ is constant. 
\end{prop}

\begin{proof}
  For the proof we need besides $D(v)$ the symmetric tensor $H(v) = -
  \frac{1}{2} (U(v) + S(v))$.  Note that $U$ and therefore also $S$
  are solutions of the Ricatti equation
  $$ \frac{d}{dt} U(\phi^t v) + U(\phi^t v)^2 + R_{\phi^t v} = 0. $$ 
  Hence, a straightforward calculation yields for all $v \in SX$
  \begin{equation} \label{eq:HGGH} (HD+DH)(\phi^tv) = S(\phi^tv)^2 -
    U(\phi^tv)^2 = \frac{d}{dt} D(\phi^t v).
  \end{equation}
  In the case $\det D(\phi^t v) = 0$ for all $t \in \R$, there is
  nothing to prove. If $\det D(\phi^t v) \neq 0$ for some $t \in \R$,
  we have
  \begin{eqnarray*}
    \frac{d}{dt} \log \det D(\phi^t v) &=& 
    \frac{1}{\det D(\phi^t v)} \tr\left( \left(\frac{d}{dt} D(\phi^t v)\right) 
      D^{-1}(\phi^t v) \right) \\
    &=& \frac{1}{\det D(\phi^t v)} \tr \left( (HD+DH)(\phi^t v) 
      D^{-1}(\phi^t v)\right) \\
    &=& \frac{2}{\det D(\phi^t v)} \tr \left( H(\phi^t v) \right) = 0,
  \end{eqnarray*}
  since $\tr H(w) = -\frac{1}{2}(\tr U(w) + \tr S(w)) = -\frac{1}{2}
  (\tr U(w) - \tr U(-w)) = 0$. This implies that $t \mapsto \det
  D(\phi^t v)$ is constant for all $t \in \R$.
\end{proof}

\subsection{$\det D(v)$ is constant along stable and unstable
  manifolds}
\label{sec:const(un)stable}

Note that the key ingredients here are Proposition \ref{prop:SUest}(b)
and Theorem \ref{thm:SUdistconv}. We first prove the following lemma.

\begin{lemma} \label{lem:DDr} Assume there is a constant $R_0 > 0$
  such that $\Vert R \Vert \le R_0$. Let $v \in SX$. Assume there is a
  constant $\rho > 0$ such that
  $$ D(\phi^t(v)) \ge \rho \id $$
  for all $t \in \R$. Then there exists a constant $a \ge 1$, depending
  only on $R_0$ such that
  $$ 0 < S_{\phi^t(v)}'(0) - S_{\phi^t(v),r}'(0) \le \frac{a}{r} \quad
  \text{and} \quad 0 < U_{\phi^t(v),r}'(0) - U_{\phi^t(v)}'(0) \le
  \frac{a}{r}
  $$
  for all $r > 0$ and all $t \in \R$.
\end{lemma} 

\begin{proof}
  Since we have $U_{w,r}'(0) = - S_{-w,r}(0)$ for all $w \in SX$, it
  suffices to prove the first assertion. Proposition
  \ref{prop:SUest}(b) yields for all $t \ge 0$
  $$ \Vert S_w(t) \Vert \le a_2=a_2(R_0,\rho), $$
  where $w = \phi^s(v)$ for some $s \in \R$. Recall from \cite[Lemma
  2.3]{Kn2} that
  $$ 
  S_w'(0) - S_{w,r}'(0) = \left( \int_0^r (S_w^* S_w)^{-1}(u)du \right)^{-1}.
  $$
  This implies for all $x \in S_{c_v(s)}X$
  \begin{eqnarray*}
  \langle (S_w'(0)-S_{w,r}'(0))x,x \rangle &\le& 
  \left\Vert \left( \int_0^r (S_w^*S_w)^{-1}(u)du \right)^{-1} \right\Vert \\ 
  &\le& \left( \int_0^r \Vert (S_w^*S_w) \Vert^{-1}(u)du \right)^{-1} \\
  &\le& \left( \int_0^r \frac{1}{a_2^2} du \right)^{-1} = \frac{a_2^2}{r},
  \end{eqnarray*}
  which yields the required estimate.
\end{proof}

Now we assume that $(X,g)$ has rank one, i.e., we have $\det D(w) > 0$
for all $w \in SX$. It suffices to show that $w \mapsto \det D(w)$ is
locally constant on $W^s(v)$. Let $v \in SX$ and $\rho > 0$ such that
$\det D(v) = 2 \rho$. Since $w \mapsto \det D(w)$ is continuous on
$SX$ by Lemma \ref{lem:contasym}, we find an open neighbourhood $U
\subset SX$ of $v$ such that $\det D(w) \ge \rho$ for all $w \in
U$. Let $v, \widetilde v \in U \cap W^s(v)$ and $\gamma: [0,1] \to U
\cap W^s(v)$ be a smooth curve with $\gamma(0)=v$ and
$\gamma(1)=\widetilde v$. We need to show that for every $\epsilon >
0$ we have $|\det D(v) - \det D(\widetilde v)| < \epsilon$. We have
$$ |\det D(v) - \det D(\widetilde v)| = |\det D(\phi^t(v)) - \det
D(\phi^t(\widetilde v))| $$ 
for all $t \in \R$ and
\begin{equation} \label{eq:dtvdttiv}
\lim_{t \to \infty }d(\phi^t(v),\phi^t(\widetilde v)) = 0, 
\end{equation}
and the convergence is exponentially because of \eqref{eq:b(t)} and
\eqref{eq:phitgamest}. 

Since our operators $D(w) = U(w) - S(w) \ge 0$ are uniformly bounded by $2 \sqrt{R_0}$ and
the determinant is a differentiable function, there is a uniform Lipschitz
constant $A > 0$ such that
$$ \vert \det D(w_1) - \det D(w_2) \vert \le A \, \Vert D(w_1) - D(w_2) \Vert. $$
Therefore, it suffices to show that, for every $\delta > 0$, there exists
$t > 0$ such that
\begin{equation} \label{eq:Ddiff}
\Vert D(\phi^t(v)) - D(\phi^t(\widetilde v)) \Vert < \delta.
\end{equation}
Let $D_r(w) = U_{w,r}'(0) - S_{w,r}'(0)$. Lemma \ref{lem:DDr} implies
$$ \Vert D(\phi^t(w)) - D_r(\phi^t(w)) \Vert \le \frac{2a}{r} $$
for all $w \in \gamma([0,1])$ and $t \in \R$. Therefore, we can choose $r > 1$
large enough such that we have
$$ \Vert D(\phi^t(w)) - D_r(\phi^t(w)) \Vert \le \frac{\delta}{3} $$
for all $w \in \gamma([0,1])$ and $t \in \R$. This implies that
\eqref{eq:Ddiff} holds if
$$ \Vert D_r(\phi^t(v)) - D_r(\phi^t(\widetilde v)) \Vert 
< \frac{\delta}{3} $$ 
for some $t > 0$. But this is a direct consequence of
\eqref{eq:dtvdttiv} and Theorem \ref{thm:SUdistconv}.

This shows that $w \to \det D(w)$ is locally and therefore also
globally constant on $W^s(v)$. Note that $w \to \det D(w)$ is also
constant on $W^u(v)$: Let $w \in W^u(v)$. Then $-w \in W^s(-v)$ and
$$ \det D(w) = \det D(-w) = \det D(-v) = \det D(v). $$

\subsection{$\det D(v)$ is constant on $SX$} 

In the case $\det D(v) = 0$ for all $v \in SX$ there is nothing to prove.
Therefore, we assume that there exists $v \in SX$ with $\det D(v) \neq 0$.

For $v \in SX$, let
\begin{eqnarray*}
  W^{0s}(v) &=& \bigcup_{t \in \R} \phi^t(W^s(v)) = \{ -\grad b_v(q) 
  \mid q \in X \}, \\
  W^{0u}(v) &=& \bigcup_{t \in \R} \phi^t(W^u(v)) = \{ \grad b_{-v}(q) 
  \mid q \in X \}.
\end{eqnarray*} 
Observe that $W^{0u}(v) = - W^{0s}(-v)$.

We define a vector $w \in SX$ to be {\em asymptotic to $v \in SX$} if
$w \in W^{0s}(v)$. Since $X$ has continuous asymptote, being
asymptotic is an equivalence relation (see \cite[Prop. 3]{Es}). We write
$v \sim w$ for asymtotic vectors $v,w \in SX$. Note
that a flow line $\phi^\R(v_1)$ can intersect a leaf $W^u(v_2)$ in at
most one vector, since the footpoint sets of these leafs are level
sets of Busemann functions and $b_v(\pi(\phi^t(w))) = b_v(\pi(w)) -t$
for asymptotic vectors $v,w \in SX$.

\begin{lemma} \label{lem:WuWueq}
  Let $v, v' \in SX$ with $\det D(v) \neq 0$. Assume that $W^u(v) = W^u(v')$
  and $v' \in W^{0s}(v)$. Then $v = v'$.
\end{lemma}

\begin{proof}
  $v' \in W^{0s}(v)$ implies that $v$ and $v'$ are asymptotic. Since
  $W^u(v) = - W^s(-v)$, $W^u(v) = W^u(v')$ implies that $-v$ and $-v'$
  are also asymptotic. Therefore, $v$ and $v'$ are bi-asymptotic.
  We have $v' \not\in \phi^{\mathbb R}(v)$, since both $v$ and $v'$
  lie in the same unstable manifold $W^u(v)$. By \cite[Thm. 1]{Es}(iv),
  there exists a central Jacobi field along $c_v$, i.e., $\ker D(v) \neq
  0$. But this contradicts to $\det D(v) \neq 0$.
\end{proof}

The assumption $\Vert R \Vert \le R_0$ implies that the
intrinsic sectional curvatures of all horospheres are also uniformly
bounded in absolute value, by the Gauss equation. Therefore, there
exists $\delta > 0$ such that for all horospheres $\mathcal H$ and
all $p \in {\mathcal H}$, the intrinsic exponential map
$\exp_{p,{\mathcal H}}: T_p{\mathcal H} \to {\mathcal H}$ is a
diffeomorphism on the ball $B_{p,{\mathcal H}}(\delta) = \{ v \in
T_p{\mathcal H} \mid \Vert v \Vert < \delta \}$.

Assume that $n = {\rm dim}\, X$. Let $v \in SX$ be a fixed vector with
$\det D(v) \neq 0$. Now, we define the following continuous map (see
Figure \ref{fig:varphimap})
$$ \varphi_v: X \times B_\delta(0) \to SX, $$
where $B_\delta(0) = \{ y \in  \R^{n-1} \mid \Vert y \Vert < \delta \}$: Choose
a smooth global orthonormal frame $Z_1 = -\grad b_v, Z_2, \dots, Z_n$
on $X$. Define
$$ \varphi_v(q,y) =  \psi^u_{Z_1(q)} \left[ \exp_{q,\pi W^u(Z_1(q))}\left(
\sum_{i=2}^n y_i Z_i(q) \right) \right] \in W^u(Z_1(q)),$$
where $\psi^u_w: \pi W^u(w) \to W^u(w)$ is defined by $\psi^u_w(q) = 
\grad b_{-w}(q)$.

  \begin{figure}[h]
  \begin{center}
    \psfrag{v}{$v$} 
    \psfrag{Wu(Z1(q))}{$W^u(Z_1(q))$}
    \psfrag{pv(q,y)}{$\varphi_v(q,y)$}
    \psfrag{Z2(q)}{$Z_2(q)$}
    \psfrag{Z3(q)}{$Z_3(q)$}
    \psfrag{q}{$q$}
    \psfrag{p}{$p$}
    \psfrag{Z1(q)=-gradbv(q)}{$Z_1(q)=-\grad b_v(q)$}
    \psfrag{ym}{$y_m$}
         \includegraphics[width=10cm]{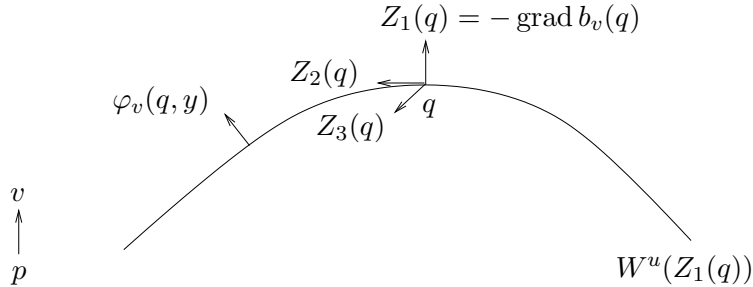}
  \end{center}
  \caption{Illustation of the map $\varphi_v: X \times B_\delta(0) \to SX$}
  \label{fig:varphimap}
  \end{figure}

We now show that $\varphi_v$ is injective: Let $\varphi_v(q,y) = \varphi_v(q',y')$. 
Then $W^u(Z_1(q)) =W^u(Z_1(q'))$ and 
$$ Z_1(q') = - \grad b_v(q') \sim v \sim -\grad b_v(q) = Z_1(q), $$
which implies $Z_1(q') \in W^{0s}(Z_1(q))$.
We conclude from the previous subsections that $\det D(Z_1(q)) = \det D(v) \neq 0$. 
Using Lemma \ref{lem:WuWueq}, we obtain $Z_1(q) = Z_1(q')$, i.e., $q=q'$. The equality
$y=y'$ follows now from the injectivity of the exponential maps and $\psi^u_w$.

Since $\dim X \times B_\delta(0) = 2n-1 = \dim SX$, we conclude that $U = \varphi_v(X \times 
B_\delta(0)) \subset SX$ is an open neighborhood of $v$, by Brouwer's domain invariance.
Moreover, $\det D(w) = \det D(v) \neq 0$ for all $w \in U$, using that $\det D$ is constant along
unstable manifolds, as well.

Now we are able to prove Theorem \ref{thm:Dconstant}.

\begin{proof}
  Without loss of generality, we assume that there exists
  a vector $v_0 \in SX$ with $\det D(v_0) = \alpha \neq 0$. Let $SX_\alpha = \{ w
  \in SX \mid \det D(w) = \alpha \}$. By continuity of $w \mapsto \det
  D_w$, the set $SX_\alpha \subset SX$ is closed. Since $v_0 \in
  SX_\alpha$, we know that $SX_\alpha$ is non-empty. The above
  arguments and Sections \ref{sec:constflow} and
  \ref{sec:const(un)stable} show for every vector $v \in SX_\alpha$
  that the open neighbourhood $\varphi_v(U)$ is contained in
  $SX_\alpha$, i.e., $SX_\alpha$ is open. Since $SX$ is connected, we
  conclude that $SX_\alpha = SX$.
  
  Since $\Vert R \Vert \le R_0$ implies that $X$ has
  bounded sectional curvature, the second fundamental forms of
  horospheres are bounded and therefore the eigenvalues of the positive
  endomorphism $D(v) = U_{v}'(0) - S_{v}'(0)$ are also uniformly bounded
  from above.The rank one assumption implies $\det D(v) = {\rm const} >
  0$. Both facts together imply that the smallest eigenvalue of $D(v)$ 
  is uniformly bounded from below by a constant $\rho > 0$.
\end{proof}

\section{Proof of the equivalences}

From now on, we assume that $(X,g)$ is asymptotically harmonic with
$\Vert R \Vert \le R_0$ and $\Vert \nabla R \Vert \le R_0'$. Our goal
is to prove Theorem \ref{thm:equivalences}. We prove each of the
implications separately.

\subsection{Rank one implies Anosov geodesic flow}

Observe first that $h = 0$ implies $\tr D(v) = \tr U(v) - \tr S(v) =
h-h = 0$.  Since $D(v)$ is positive, this implies $D(v) = 0$ and $\det
D(v) = 0$ which contradicts to $\rank(X) = 1$. Now we assume that
${\rm rank}(X) = 1$ and, therefore, $D(v) \ge \rho > 0$, by Theorem
\ref{thm:Dconstant}. By \cite[Theorem, p. 107]{Bo} this implies that
the geodesic flow is Anosov.

\subsection{Anosov geodesic flow implies Gromov hyperbolicity}

Recall that a geodesic metric space is called {\em Gromov hyperbolic}
if there exists $\delta > 0$ such that every geodesic triangle is
$\delta$-thin, i.e., every side of the triangle is contained in the
union of the $\delta$-tubular neighborhoods of the other two sides.

Assume now that the geodesic flow $\phi^t:SX \to SX$ is Anosov with
respect to the Sasaki metric. For $v \in SX$ consider the normal
Jacobi tensor along $c_v$ with $A_v(0) = 0$ and $A_v'(0) = \id$. Then
the Anosov property implies (see \cite[p. 113]{Bo})
$$
\|A_v(t)x \|\ge c e^{\alpha t} \Vert x \Vert
$$
for $t \ge 1$ with suitable constants $c, \alpha > 0$.
Consider two distinct geodesic rays $c_1:[0, \infty) \to X$ and
$c_2:[0, \infty) \to X$ with $c_1(0) =c_2(0) = p$ and define
\begin{eqnarray*}
d^q_t(c_1(t), c_2(t)) &:= &\inf \{ L(\gamma) \mid  \;
\gamma:[a,b] \to  X \setminus B(p,t) \\
&& \text{ a piecewise smooth curve
joining } \; c_1(t) \text{ and } c_2(t) \},
\end{eqnarray*}
where $B(p,t) = \{ q \in X \mid d(p,q) < t \}$. Let $t \ge 1$ and
$\gamma: [0,1] \to X \backslash B(p,t)$ be a curve connecting $c_1(t)$
and $c_2(t)$. Let $v_1 = c_1'(0) \in S_pX$ and $v_2 = c_2'(0) \in
S_pX$. Then $\gamma(s) = \exp_p(r(s)v(s))$ with $r: [0,1] \to [t,
\infty)$ and $v: [0,1] \to S_pX$ and
\begin{eqnarray*}
\gamma'(s) &=&  D\exp_p(r(s)v(s))\left(r'(s)v(s)+r(s)v'(s)\right) \\
&=& r'(s) c_{v(s)}'(r(s)) + A_v(r(s))v'(s).
\end{eqnarray*}
Since $c_{v(s)}'(r) \bot A_v(r)v'(s)$, we conclude that
$$ \Vert \gamma'(s) \Vert \ge c e^{\alpha r(s)} \Vert v'(s) \Vert. $$
This implies that
$$ L(\gamma) = \int_0^1 \Vert \gamma'(s) \Vert ds \ge c e^{\alpha t} 
\angle(v_1,v_2), $$
and therefore
\begin{equation} \label{eq:expdiv}
\liminf_{t \to \infty} \frac { \log d^q_t(c_1(t), c_2(t))}{t} \ge c_0 \alpha
\end{equation}
with a suitable constant $c_0>0$. This implies, using \cite[Chapter III, 
Prop. 1.26]{BH} that $X$ is Gromov hyperbolic. (Note that the condition
there is $\liminf_{t \to \infty} \frac{d^q_t(c_1(t), c_2(t))}{t} = \infty$, which
is a priori weaker than \eqref{eq:expdiv}. In fact, both conditions are
equivalent to Gromov hyperbolicity, see \cite[Chapter III, Prop. 1.25]{BH}.)
  
\subsection{Gromov hyperbolicity implies purely exponential
  volume growth with $h=h_{vol}$}

We like to note first that simply connected Riemannian manifolds
$X$ without conjugate points which are Gromov hyperbolic spaces
{\em admitting compact quotients} have purely exponential volume
growth (see \cite[Thm. 7.2]{Coor}). Here we consider the special case
of an asymptotic harmonic manifold without the additional assumption
that $X$ admits a compact quotient.

\begin{dfn}
  Let $X$ be a Riemannian manifold with $h_{vol} = h_{vol}(X) >
  0$. Then $X$ has {\em purely exponential volume growth with
    growth rate $h_{vol}$} if, for every $p \in X$, there exists a
  constant $C = C(p) \ge 1$ with
  $$ \frac{1}{C} e^{h_{vol} r} \le \vol B_r(p) \le C e^{h_{vol} r} 
  \quad \text{for all $r \ge 1$.} $$ 
\end{dfn}

We first prove the following general lemma.

\begin{lemma} \label{lem:Gromovexpvol}
  Let $X$ be a $\delta$-hyperbolic space without conjugate points and
  bounded curvature. Then the volume of any geodesic sphere
  grows exponentially. In particular, we have $h_{vol}(X) > 0$.
\end{lemma}

\begin{proof}
  Fix $p \in X$ and geodesic rays $c_1, c_2: [0,\infty) \to X$ with $c_1(0)=c_2(0)$.
  As remarked above, Gromov hyperbolicity implies
  $$
  \liminf_{t \to \infty} \frac { \log d^q_t(c_1(t), c_2(t))}{t} \ge c(\delta),
  $$
  where $c(\delta) > 0$ depends only on the Gromov constant $\delta$.
  In particular, there exists $t_0 > 0$ such that for all $t \ge t_0$
  $$ d_{S_p(t)}(c_1(t),c_2(t)) \ge e^{t c(\delta)/2}, $$
  where $d_{S_p(t)}$ is the intrinsic distance in the sphere $S_p(t) \subset X$.
  Let $\gamma_t : [0,l(t)] \to S_p(t)$ be a minimal geodesic in $S_p(t)$ connnecting
  $c_1(t)$ and $c_2(t)$. The $1/4$-balls in $S_p(t)$ with centers $\gamma_t(k)$
  and $k \in \Z \cap  [0,l(t)]$ are pairwise disjoint. Lemma \ref{lem:AAm1} implies that the
  second fundamental forms of $S_p(t)$ are bounded by a universal constant for
  all $t \ge t_0 > 0$. Using the Gauss equation, this implies that the curvatures of 
  the spheres $S_p(t)$ are uniformly bounded for $t \ge t_0$, as well. Therefore,
  the $1/4$-balls in $S_p(t)$ have a uniform lower volume bound $A_0 > 0$. Hence,
  we have
  $$ \vol(S_t(p)) \ge A_0 (e^{t c(\delta)/2} - 1) $$
  for all $t \ge t_0$. This finishes the proof of the lemma.
\end{proof}

\begin{lemma} \label{lem:lowexp} Let $(X,g)$ be an asymptotically
  harmonic manifold. Then, for all $p \in X$, there
  exists a constant $C_1(p) > 0$ such that
  $$ \frac{\vol S_r(p)}{e^{hr}} \ge C_1(p) \quad 
  \text{for all $r \ge 1$.}  $$
  In particular, we have
  $$ h \le h_{vol}(X). $$
\end{lemma}

\begin{proof}
  As in the proof of \cite[Cor. 25]{Kn2}, we have for all $v \in SX$
  $$ \frac{\det A_v(t)}{\det U_v(t)} = \frac{\det A_v(t)}{e^{ht}} = 
  \frac{1}{\det(U(v) - S_{v,t}'(0))}, $$ 
  which implies
  \begin{equation} \label{eq:vol/ehr} 
  \frac{\vol S_r(p)}{e^{hr}} = 
  \int_{S_pX} \frac{1}{\det(U(v) - S_{v,r}'(0))} d\theta_p(v). 
  \end{equation}
  Using $U(v) - S_{v,t_1}'(0) \ge U(v) - S_{v,t_2}'(0) > 0$ for all $0
  < t_1 < t_2$, we obtain
  $$ \frac{\vol S_r(p)}{e^{hr}} \ge 
  \int_{S_pX} \frac{1}{\det(U(v) - S_{v,1}'(0))} d\theta_p(v). $$ 
  Continuous asymptote implies the continuity of $v \mapsto U(v) -
  S_{v,1}'(0)$. This yields the existence of a constant $a > 0$ such
  that $ \det(U(v) - S_{v,1}'(0)) \ge a$ for all $v \in S_pX$ and
  implies the statement in the lemma.
\end{proof}

Recall the following result in \cite[Cor. 4.6]{Kn2}.

\begin{prop} \label{prop:coneinhoroball}
  Let $X$ be a simply connected $\delta$-hyperbolic manifold without
  conjugate points. Consider for $v \in S_pX$, $\ell = \delta +1$ and
  $r >0$ the spherical cone in $X$ given by
  $$
  A_{v, \ell}(r) := \{ c_w(t) \mid 0 \le t \le r , w \in S_pX, d(c_v(
  \pm \ell), c_w( \pm \ell)) \le 1 \}.
  $$
  Then, for $\rho= 4 \delta +2$ the set $A_{v, \ell}(r) $ is contained
  in
  \begin{equation*}
    \begin{split}
      H_{v,\rho}(r) :=\{ c_q(t) \mid & -\rho/2 \le t \le r , \; c_q \;
      \text{is an integral curve of}\\
      & \grad b_{-v}\; \text{with} \; c_q(0) = q \in b_{-v}^{-1}(0)
      \cap B(p,\rho) \}.
    \end{split}
  \end{equation*}
\end{prop}

This useful result has the following consequence.

\begin{cor} \label{cor:hvolleqh}
  Let $(X,g)$ be a Gromov hyperbolic asymptotically
  harmonic manifold and $p \in X$. Then there exists
  a constant $C_2(p) > 0$ such that
  $$ \vol B_r(p) \le C_2(p) \int_{-\rho/2}^{r} e^{hs} ds, $$
  where $\rho$ is defined as in Proposition \ref{prop:coneinhoroball}.
  In particular, we have
  $$ h_{vol}(X) \le h. $$
\end{cor}

\begin{proof}
Let $p \in X$. Choose $l = \delta + 1$. Then we have
$$ S_pX = \bigcup_{v \in S_pX} U_{v,\ell}(r), $$
with the open sets
$$ U_{v,\ell}(r) = \{ w \in S_pX \mid d(c_v( \pm \ell), c_w( \pm \ell)) < 1 \}. $$
Since $S_pX$ is compact, we find finitely many vectors 
$v_1,\dots,v_k \in S_pX$ with
$$ S_pX = \bigcup_{j=1}^k U_{v_j,\ell}(r), $$
which implies for $\rho = 4 \delta + 2$
$$ B_r(p) \subset \bigcup_{j=1}^k A_{v_j,\ell}(r) \subset 
\bigcup_{j=1}^k H_{v_j,\rho}(r). $$
Using
$$ \vol(H_{v,\rho}(r) ) = \int\limits_{-\rho/2}^{r} e^{hs} ds 
\vol_{0}( b_v^{-1}(0) \cap B_\rho(p))
$$
where $\vol_{0}$ denotes the induced volume on the horosphere $b_v^{-1}(0)$,
we conclude
$$ \vol B_r(p) \le \left( \sum_{j=1}^k \vol_{0}( b_{v_j}^{-1}(0) \cap B_\rho(p)) 
\right) \int\limits_{-\rho/2}^{r} e^{hs} ds. $$
Setting $C_2(p) = \sum_{j=1}^k \vol_{0}( b_{v_j}^{-1}(0) \cap B_\rho(p))$
proves the first part of the corollary. The inequality $h_{vol}(X) \le h$ follows
then from the definition of $h_{vol}(X)$.
\end{proof}
 
Now we prove the implication claimed in this subsection.

\begin{prop}
  Let $(X,g)$ be a Gromov hyperbolic asymptotically harmonic space with with
  bounded curvature. Then $X$ has purely exponential volume growth with
  $h = h_{vol}$.
\end{prop}

\begin{proof}
  Gromov hyperbolicity implies $h_{vol}(X) > 0$, by Lemma \ref{lem:Gromovexpvol}.
  Lemma \ref{lem:lowexp} and Corollary \ref{cor:hvolleqh} together yield $h = h_{vol}(X)$. 
  Moreover, we derive from Corollary \ref{cor:hvolleqh} that
  $$ \vol B_r(p) \le \frac{C_2(p)}{h} e^{hr}. $$ 
  The lower volume estimate
  follows from Lemma \ref{lem:lowexp}: For $r \ge 2$ we have
  $$ \frac{\vol B_r(p)}{e^{hr}} \ge \frac{\int_{r-1}^r \vol S_t(p) dt}{e^{hr}} 
  \ge \frac{\vol S_{t_0}(p)}{e^{hr}} \ge \frac{C_1(p)}{e^h}, $$ 
  for some $t_0 \in [r-1,r]$. This finishes the proof of purely exponential volume growth.
\end{proof}

\subsection{Purely exponential volume growth with $h=h_{vol}$
  implies rank one} Finally, we show the remaining implication of Theorem
\ref{thm:equivalences}. This closes the chain of implications and
finishes the proof that all four properties listed in (a), (b), (c)
and (d) are equivalent.

Assume that $(X,g)$ is asymptotically harmonic with purely exponential
volume growth $h = h_{vol}$.  We have
$$ \int_{r-1}^{r} e^{ht} \int_{S_pX} \frac{1}{\det(U(v)-S_{v,t}'(0))} d\theta_p(v) dt \le \vol(B_r(p)). $$
This implies
$$ \frac{1}{e} \int_{r-1}^{r} \int_{S_pX} \frac{1}{\det(U(v)-S_{v,t}'(0))} d\theta_p(v) dt \le \frac{\vol(B_r(p))}{e^{hr}} \le C(p) $$
for some constant $C(p) > 0$. Assume that $\det(U(v)-S_{v,t}'(0)) \to 0$ for all $v \in S_pX$. Then,
because of monotonicity and Dini, we know that this convergence is
uniform. This is in contradiction to the above inequality. Therefore,
there exist $v \in S_pX$ with $\det(U(v)-S(v)) \neq 0$ and $(X,g)$ has
rank one.

\section{Asymptotically harmonic manifolds with bounded asymptote}

The notion of bounded asymptote was first introduced by Eschenburg in \cite[Section 4]{Es}. Examples
of manifolds of bounded asymptote are manifolds with nonpositive curvature or, more generally, manifolds
with no focal points.

\begin{dfn}
  Let $(X,g)$ be a complete, simply connected Riemannian manifold without conjugate points.
  $X$ is called a manifold {\em of bounded asymptote} if there exists a constant $A \ge 1$ such
  that 
  \begin{equation} \label{eq:bdasym} 
  \Vert S_v(t) \Vert \le A \quad \forall\, t \ge 0, \quad \forall\; v \in SX. 
  \end{equation}
\end{dfn}

\begin{lemma}
  The bounded asymptote property \eqref{eq:bdasym} implies
  $$ \Vert U_v(t) \Vert \ge \frac{1}{A} \quad \forall\, t \ge 0, \quad \forall\; v \in SX. $$
\end{lemma}

\begin{proof}
  Letting $x \to \infty$, we conclude from \eqref{eq:Stransform} 
  $$ S_{\phi^t v}(y)= S_v(y+t)S_v^{-1}(t). $$
  Using $S_v(t) = U_{-v}(-t)$, we obtain
  $$ S_{-\phi^t v}(s) = U_v(t-s)U_v^{-1}(t). $$
  This implies
  $$ 1 = \Vert S_{-\phi^t v}(t) U_v(t) \Vert \le A \; \Vert U_v(t) \Vert, $$
  finishing the proof.
\end{proof}

\begin{rmk}
  Rank one asymptotically harmonic manifolds with $\Vert R \Vert \le R_0$ and $\Vert \nabla R \Vert \le R_0'$
  are manifolds of bounded asymptote by Proposition \ref{prop:SUest}.
\end{rmk}

Next, we discuss relations between the geometrically defined constants $h, h_{vol}(X)$ and the {\em Cheeger constant} $h_{Cheeg}(X)$, defined as
$$ h_{Cheeg}(X) = \inf_{K \subset X} \frac{{\rm area}(\partial K)}{{\rm vol}(K)}, $$
where $K$ ranges over all connected, open submanifolds of $X$ with compact closure and smooth boundary.

\begin{prop} \label{prop:hvolhcheegh}
  Let $(X,g)$ be an asymptotically harmonic manifold. Then we have
  $$ h_{vol}(X), h_{Cheeg}(X) \ge h. $$
\end{prop}

\begin{proof} 
  The inequality $h_{vol}(X) \ge h$ was already stated in Lemma \ref{lem:lowexp}. For the proof of $h_{Cheeg}(X) \ge h$
  let $K \subset X$ be a set as described above. Choosing a Busemann function $b_v$, we have $\Delta b_v = h$ and
  obtain via Gauss' divergence theorem and $\Vert \grad b_v \Vert = 1$,
  $$ h \vol(K) = \int_K \Delta b_v(x) dx = \int_{\partial K} \langle \grad b_v,\nu \rangle dx \le {\rm area}(\partial K), $$
  where $\nu$ is the outward unit normal vector of $\partial K$ in $X$. 
\end{proof}

Even though we proved in the previous section that $h = h_{vol}(X)$ for Gromov
hyperbolic asymptotically harmonic spaces $X$ with bounded curvature, we do not know
whether this holds for general asymptotically harmonic manifolds. However, a sufficient condition
for $h = h_{vol}(X) = h_{Cheeg}(X)$ is that $X$ is asymptotically harmonic and has bounded
asymptote.

\begin{theorem}
  Let $(X,g)$ be asymptotically harmonic and of bounded asymptote. Then we have
  $$ h = h_{vol}(X) = h_{Cheeg}(X). $$
  In particular, this equality holds for all rank one asymptotically
  harmonic manifolds with $\Vert R \Vert \le R_0$ and $\Vert \nabla R
  \Vert \le R_0'$.
\end{theorem} 

\begin{proof}
  The bounded asymptote property implies that we have
  $$
  \frac{1}{A^2 t} \le \langle (U(v) - S_{v,t}'(0))x, x \rangle,
  $$
  for all unit vectors $x \in v^\bot$ (see the proof of \cite[Prop. 5.2]{Kn2}). 
  This implies
  $$ \det (U(v) - S_{v,t}'(0)) \ge \frac{1}{A^{2n-2} t^{n-1}}., $$
  and we obtain with \eqref{eq:vol/ehr}
  $$ \frac{\vol S_r(p)}{e^{hr}} \le \int_{S_pX} A^{2n-2} r^{n-1} 
  d\theta_p(v) = \omega_{n-1} A^{2n-2} r^{n-1}, $$ where
  $\omega_{n-1}$ is the volume of the Euclidean unit sphere of
  dimension $n-1$.This implies $\vol S_r(p) \le C r^{n-1} e^{hr}$ and,
  therefore, $h_{vol}(X) \le h$. Together with Lemma
  \ref{lem:lowexp} we obtain $h = h_{vol}(X)$.
  
  Next we prove $h_{Cheeg}(X) \le h$: Let $g(r) = \frac{\vol S_r(p)}{\vol B_r(p)}$.
  We will show that $g(r) \to h$ for $r \to \infty$ which implies $h_{Cheeg}(X) \le h$.
  We have with l'Hospital
  \begin{eqnarray*}
  \lim_{r \to \infty} g(r) &=& \lim_{r \to \infty} \frac{\int_{S_pX} \det A_v(r)\; d\theta_p(v)}{\int_0^r 
  \int_{S_pX} \det A_v(s)\; d\theta_p(v) \; ds} \\
  &=& \lim_{r \to \infty} \frac{\int_{S_pX} \tr (A_v'(r) A_v^{-1}(r)) \det A_v(r)\; d\theta_p(v)}
  {\int_{S_pX} \det A_v(r)\; d\theta_p(v)},
  \end{eqnarray*}
  provided the last limit exists. (We used the notion $d\theta$ for the canonical volume element 
  of the unit sphere $S_pX$). Note that $A_v'(r) A_v^{-1}(r) = U_{\phi^r v,r}'(0)$. Since 
  $\Vert U_{w,r}'(0) - U_w'(0) \Vert \le \frac{A^2}{r}$ (see, for instance, \cite[bottom of p. 686]{Kn2}),
  we conclude
  $$ 0 \le \tr U_{w,r}'(0) - h \le (n-1) \frac{A^2}{r} $$
  for all $w \in SX$ and $r \ge 0$. Therefore, $\tr (A_v'(r)
  A_v^{-1}(r)) \to h$ and the convergence is uniformly, which implies
  that the last limit above exists and is equal to $h$. This, together
  with Proposition \ref{prop:hvolhcheegh} above, implies that
  $h_{Cheeg}(X) = h$.
\end{proof}

\begin{remark}
  It was shown by Zimmer in the proof of \cite[Cor. 49]{Zi} that $
  h_{vol}(X) = h $ also holds in the case that $(X,g)$ is
  asymptotically harmonic {\em admitting compact quotients}. Equality
  of $h, h_{vol}(X)$ and $h_{Cheeg}(X)$ also holds for all {\em
    noncompact harmonic manifolds} $X$ without additional conditions
  (see \cite[Theorem 5.1]{PeSa}). Moreover, the agreement of these three
  geometric constants implies (see \cite[Corollary 5.2]{PeSa}) that the
  bottom of the spectrum and of the essential spectrum of the
  Laplacian $\Delta_X$ coincide and $\lambda_0(X) = \lambda_0^{\rm ess}(X) = \frac{h^2}{4}$.
\end{remark}

\end{document}